\documentclass[oneside,10pt]{amsart}
\usepackage{amsfonts}
\usepackage{amsthm}
\usepackage{amsmath}
\usepackage{amscd}
\usepackage[latin2]{inputenc}
\usepackage{t1enc}
\usepackage[mathscr]{eucal}
\usepackage{indentfirst}
\usepackage{graphicx}
\usepackage{graphics}
\usepackage{pict2e}
\usepackage{epic}
\numberwithin{equation}{section}
\usepackage[margin=2.9cm]{geometry}
\usepackage{epstopdf}
\newtheorem{theorem}[subsection]{Theorem}
\newtheorem{lemma}[subsection]{Lemma}
\newtheorem{proposition}[subsection]{Proposition}

\newtheorem{corollary}[subsection]{Corollary}
\theoremstyle{definition}
\newtheorem{remark}[subsection]{Remark}
\newtheorem{definition}[subsection]{Definition}
\newtheorem{ex}[subsection]{Example}

\theoremstyle{property}

\makeatletter
\newcommand{\vast}{\bBigg@{2}}
\newcommand{\Vast}{\bBigg@{3}}
\makeatother

\begin{document}

\title{Some Hermite-Hadamard Type Inequalities For Harmonically $(s,m)$-convex functions in Second Sense}
\author{Imran Abbas Baloch, $\dot{I}$mdat $\dot{I}$\c{s}can}
\address{Imran Abbas Baloch\\Abdus Salam School of Mathematical Sciences\\
GC University, Lahore, Pakistan}\email{iabbasbaloch@gmail.com\\iabbasbaloch@sms.edu.pk}

\address{$\dot{I}$mdat $\dot{I}$\c{s}can\\ Department of Mathematics, Faculty of Arts and Sciences,\\
 Giresun University, 28200, Giresun, TURKEY}
\email{imdat.iscan@giresun.edu.tr}



\begin{abstract} Authors introduce the concept of harmonically $(s,m)$-convex functions in second sense in \cite{II}.In this article, we establish some Hermite-Hadamard type inequalities of this class of functions.
\end{abstract}

\maketitle

\section{\bf{Introduction}}
Let $f:I \subseteq \mathbb{R} \rightarrow \mathbb{R}$ be a convex function and $a,b \in I$ with $a < b$. Then following inequality
\begin{equation}\label{1}
f \vast( \frac{a + b}{2} \vast) \leq \frac{1}{b - a} \int_{a}^{b} f(x) dx \leq \frac{f(a) + f(b)}{2}
\end{equation}
holds. This double inequality is known in the literature as the Hermite-Hadamard integral inequality for convex function. Note that some of the classical inequalities for mean can be derived from (\ref{1}) for the appropriate particular selection of mapping $f$.Both inequalities in (\ref{1}) hold in the reverse direction if $f$ is concave.\\
In \cite{Ak}, Imdat Iscan introduced the concept of harmonically convex function, and established a variant od Hermite-Hadamard type inequalities which holds for these classes of functions as follows:\\
\begin{definition}
Let $I \subset \mathbb{R}/ \{0\}$ be a real interval. A function $f:I \rightarrow \mathbb{R}$ is said to be harmonically convex, if
\begin{equation}\label{2}
f\vast( \frac{xy}{tx + (1 - t)y} \vast) \leq tf(y) + (1 - t)f(x)
\end{equation}
for all $x,y \in I$ and $t \in [0,1]$. If inequality in (\ref{2}) is reversed, then $f$ is said to be harmonically concave.
\end{definition}
\begin{theorem}(see \cite{Ak})
Let $f:I \subset \mathbb{R}/ \{0\} \rightarrow \mathbb{R}$ be harmonically convex and $a,b \in I$ with $a < b$. If $f \in L[a,b]$, then following inequalities hold
\begin{equation}
f \big( \frac{2ab}{a + b}\big) \leq \frac{ab}{ b - a}\int_{a}^{b} \frac{f(x)}{x^{2}}dx \leq \frac{f(a) + f(b)}{2}.
\end{equation}
\end{theorem}
\begin{theorem}(see \cite{Ak}).\label{I1}
Let $f:I \subset (0,\infty) \rightarrow \mathbb{R}$ be a differentiable on $I^{\circ}$,$a,b \in I$ with $a < b$ and $f' \in L[a,b]$. If $|f'|^{q}$ is harmonically convex on $[a,b]$ for $q \geq 1$, then
 \begin{equation}
 \vast| \frac{f(a) + f(b)}{2} - \frac{ab}{b - a} \int_{a}^{b} \frac{f(x)}{x^{2}}dx \vast| \leq \frac{ab(b - a)}{2} \lambda_{1}^{1 - \frac{1}{q}}\big(\lambda_{2} |f'(a)|^{q} + \lambda_{3} |f'(b)|^{q}\big)^{\frac{1}{q}},
 \end{equation}
  where
  $$\lambda_{1} = \frac{1}{ab} - \frac{2}{(b - a)^{2}}\ln \vast( \frac{(a + b)^{2}}{4ab}\vast), $$
  $$ \lambda_{2} = \frac{-1}{b(b - a)} + \frac{3a + b}{(b - a)^{3}} \ln \vast( \frac{(a + b)^{2}}{4ab}\vast), $$
  $$ \lambda_{3} =  \frac{1}{a(b - a)} + \frac{3a + b}{(b - a)^{3}} \ln \vast( \frac{(a + b)^{2}}{4ab}\vast) .$$
\end{theorem}
\begin{theorem}(see \cite{Ak}).\label{I2}
Let $f:I \subset (0,\infty) \rightarrow \mathbb{R}$ be a differentiable function on $I^{\circ}$, $a,b \in I$ with $a < b$ and $f' \in [a,b]$. If $|f'|^{q}$ is harmonically convex on $[a,b]$ for $q > 1$, $\frac{1}{p} + \frac{1}{q} = 1$, then
\begin{equation}
\vast| \frac{f(a) + f(b)}{2} - \frac{ab}{b - a} \int_{a}^{b} \frac{f(x)}{x^{2}}dx \vast| \leq \frac{ab(b - a)}{2} (\frac{1}{p +1})^{\frac{1}{p}}\big(\mu_{1} |f'(a)|^{q} + \mu_{2} |f'(b)|^{q}\big)^{\frac{1}{q}},
\end{equation}
where
$$ \mu_{1} = \frac{[a^{2 - 2q} + b^{1 - 2q}[(b - a)(1 - 2q) - a]]}{2 (b - a)^{2} (1 - q)(1 - 2q)},$$

$$ \mu_{2} = \frac{[b^{2 - 2q} - a^{1 - 2q}[(b - a)(1 - 2q) + b]]}{2 (b - a)^{2} (1 - q)(1 - 2q)}. $$
\end{theorem}
In \cite{mnpj}, Imdat Iscan introduced the concept of harmonically $s$-convex function in second sense as follow:
\begin{definition}
A function $f:I \subset \mathbb{R}/ \{0\} \rightarrow \mathbb{R}$ is said to be harmonically $s$-convex in second sense , if
\begin{equation}\label{3}
f\vast( \frac{xy}{tx + (1 - t)y} \vast) \leq t^{s}f(y) + (1 - t)^{s}f(x)
\end{equation}
for all $x,y \in I$ and $t \in [0,1]$. If inequality in (\ref{3}) is reversed, then $f$ is said to be harmonically $s$-concave.
\end{definition}
\begin{remark}
Note that for $s = 1$, harmonically $s$-convexity reduces to ordinary harmonically convexity.
\end{remark}
In \cite{FU}, Feixiang Chen and Shanhe Wu generalized Hermite-Hadamard type inequalities given in \cite{Ak} which hold for harmonically $s$-convex functions in second sense.
\begin{theorem}(see \cite{FU}).\label{FS1}
Let $f:I \subset (0,\infty) \rightarrow \mathbb{R}$ be a differentiable on $I^{\circ}$ such that $f' \in L[a,b]$, where $a,b \in I^{\circ}$ with $a < b$. If $|f'|^{q}$ is harmonically $s$- convex in second sense on $[a,b]$ for some fixed $s \in (0,1]$, $q \geq 1$, then
 $$\vast| \frac{f(a) + f(b)}{2} - \frac{ab}{b - a} \int_{a}^{b} \frac{f(x)}{x^{2}}dx \vast| \leq \frac{ab(b - a)}{2} C_{1}^{1 - \frac{1}{q}}(a,b)[C_{2}(s;a,b) |f'(a)|^{q} + C_{3}(s;a,b) |f'(b)|^{q}]^{\frac{1}{q}},$$
 where
 $$ C_{1}(a,b) = b^{-2}\vast( {}_{2}F_{1} (2,2;3,1 - \frac{a}{b}) - {}_{2}F_{1}(2,1;2,1 - \frac{a}{b}) + \frac{1}{2} {}_{2}F_{1}(2,1;3, \frac{1}{2}(1 - \frac{a}{b})   \vast)$$
 $$ C_{2}(s,a,b) = b^{-2} \vast( \frac{2}{s + 2} {}_{2}F_{1} (2,s + 2;s + 3,1 - \frac{a}{b}) - \frac{1}{s + 1} {}_{2}F_{1}(2,s + 1;s + 2,1 - \frac{a}{b}) $$
 $$+ \frac{1}{2^{s}(s + 1)(s + 2)} {}_{2}F_{1} (2,s + 1;s + 3,\frac{1}{2}(1 - \frac{a}{b})) \vast)$$

$$ C_{3}(s,a,b) = b^{-2} \vast( \frac{2}{(s + 1)(s + 2)} {}_{2}F_{1} (2,2;s + 3,1 - \frac{a}{b}) - \frac{1}{s + 1} {}_{2}F_{1}(2,1;s + 2,1 - \frac{a}{b}) $$
 $$+ \frac{1}{2} {}_{2}F_{1} (2,1;3,\frac{1}{2}(1 - \frac{a}{b})) \vast)$$
\end{theorem}
\begin{remark}
Note that for $s = 1$, $ C_{1}(a,b) = \lambda_{1}$, $C_{2}(1,a,b) = \lambda_{2}$ and $C_{3}(1,a,b) = \lambda_{3}$. Hence, Theorem \ref{I1} is particular case of theorem \ref{FS1} for $s = 1$.
\end{remark}
\begin{theorem}(see \cite{FU})\label{FS2}
Let $f:I \subset (0,\infty) \rightarrow \mathbb{R}$ be a differentiable on $I^{\circ}$ such that $f' \in L[a,b]$, where $a,b \in I$ with $a < b$. If $|f'|^{q}$ is harmonically $s$- convex in second sense on $[a,b]$ for some fixed $s \in (0,1]$, $q > 1$, then
 $$\vast| \frac{f(a) + f(b)}{2} - \frac{ab}{b - a} \int_{a}^{b} \frac{f(x)}{x^{2}}dx \vast| $$
 $$\leq \frac{a(b - a)}{2b} (\frac{1}{p + 1})^{\frac{1}{p}} \vast[\frac{1}{s + 1} [{}_{2}F_{1}(2q,s + 1;s + 2,1 - \frac{a}{b})|f'(b)|^{q} + {}_{2}F_{1}(2q,1;s + 2,1 - \frac{a}{b}) |f'(a)|^{q}]\vast]^{\frac{1}{q}}$$
\end{theorem}
\begin{remark}
Note that for $s = 1$
$$ \mu_{1} = \frac{1}{2b^{2q}}. {}_{2}F_{1}(2q,2,3,1 - \frac{a}{b}), $$
and
$$ \mu_{2} =  \frac{1}{2b^{2q}}. {}_{2}F_{1}(2q,1,3,1 - \frac{a}{b}).$$
Hence, Theorem \ref{I2} is particular case of theorem \ref{FS2} for $s = 1$.
\end{remark}
In (\cite{MP}), Jaekeun Park considered the class of $(s,m)$-convex functions in second sense. This class of function is defined as follow
\begin{definition}
For some fixed $s \in (0,1]$ and $m \in [0,1]$ a mapping $f:I \subset [0,\infty) \rightarrow \mathbb{R}$ is said to be $(s,m)$-convex in the second sense on $I$ if
$$  f(tx + m(1 - t)y) \leq t^{s}f(x) + m (1 - t)^{s}f(y)  $$
holds, for all $x,y \in I$ and $t \in [0,1]$
\end{definition}
In (\cite{Akk}), Imdat Iscan introduced the concept of harmonically $(\alpha,m)$-convex functions and established some Hermite-Hadamard type inequalities for this class of function. This class of functions is defined as follow
\begin{definition}
The function $f:(0,\infty) \rightarrow \mathbb{R}$ is said to be harmonically $(\alpha,m)$-convex, where $\alpha \in [0,1]$ and $m \in (0,1]$, if
\begin{equation}\label{4}
 \big(\frac{mxy}{mty + (1 - t)x}\big) = f \big( (\frac{t}{x} + \frac{1 - t}{my})^{-1} \big) \leq t^{\alpha} f(x) + m (1 - t^{\alpha}) f(y)
\end{equation}
for all $x,y \in (0,\infty)$ and $t \in [0,1]$. If the inequality in (\ref{4}) is reversed, then $f$ is said to be harmonically $(\alpha,m)$-concave.
\end{definition}

In \cite{II} , authors introduce the concept of Harmonically $(s,m)$-convex functions in second sense which generalize the notion of Harmonically convex and Harmonically $s$-convex functions in second sense introduced by Imdat Iscan in \cite{Ak},\cite{mnpj}.\\
In this paper, we establish some results connected with the right side of new inequality similar to (\ref{1}) for this class of functions such that results given by Imdat Iscan \cite{Ak} , Feixiang Chen and Shanhe Wu \cite{FU} are obtained for the particular values of $s,m$.

\begin{definition}
The function $f: I \subset (0, \infty) \rightarrow \mathbb{R}$ is said to be harmonically $(s,m)$-convex in second sense, where $s \in (0,1]$ and $m \in (0,1]$ if
$$f \big(\frac{mxy}{mty + (1 - t)x}\big) = f \big( (\frac{t}{x} + \frac{1 - t}{my})^{-1} \big) \leq t^{s} f(x) + m (1 - t)^{s} f(y)$$
$\forall x, y \in I$ and $t \in [0,1]$.
\end{definition}
\begin{remark}
Note that for $s = 1$, $(s,m)$-convexity reduces to harmonically $m$-convexity and for $m = 1$, harmonically $(s,m)$-convexity reduces to harmonically $s$-convexity in second sense (see \cite{mnpj}) and for $s,m = 1$, harmonically $(s,m)$-convexity reduces to ordinary  harmonically convexity (see \cite{Ak}).
\end{remark}
\begin{proposition}\label{PP1}
Let $f:(0,\infty) \rightarrow \mathbb{R}$ be a function\\
a)  if $f$ is $(s,m)$-convex function in second sense and non-decreasing, then$f$ is harmonically $(s,m)$-convex function in second sense.\\
b)  if $f$ is harmonically $(s,m)$-convex function in second sense and non-increasing, then $f$ is $(s,m)$-convex function in second sense.\\
\begin{proof}
For all $t \in [0,1]$, $m \in (0,1)$ and $x,y \in I$, we have
$$t (1 - t) (x - my)^{2} \geq 0 $$
Hence, the following inequality holds
\begin{equation}\label{P1}
\frac{mxy}{mty + (1 - t)x} \leq tx + m(1 - t)y
\end{equation}
By the inequality (\ref{P1}), the proof is completed.
\end{proof}
\end{proposition}
\begin{remark}
According to proposition \ref{PP1}, every non-decreasing $(s,m)$-convex function in second sense is also harmonically $(s,m)$-convex function in second sense.
\end{remark}
\begin{ex}(see\cite{BK})
Let $0 < s < 1$ and $a, b, c \in \mathbb{R}$, then function $f:(0,\infty) \rightarrow \mathbb{R}$ defined by
$$  f(x) = \left\{
                                                            \begin{array}{ll}
                                                              a, & \hbox{$ x = 0$} \\
                                                              b x^{s} + c, & \hbox{$x > 0$}
                                                            \end{array}
                                                          \right.    $$
is non-decreasing $s$-convex function in second sense for $ b \geq 0$ and $ 0 \leq c \leq a$. Hence, by proposition \ref{PP1}, $f$ is harmonically $(s,1)$-convex function.
\end{ex}
\begin{proposition}
Let $s \in [0,1]$, $m \in (0,1]$, $f : [a, mb]\subset(0,\infty)  \rightarrow \mathbb{R}$, be an increasing function and $g : [a,mb] \rightarrow [a,mb]$, $g(x) = \frac{mab}{a + mb - x}$, $a < mb$. Then $f$ is harmonically $(s,m)$-convex in second sense on $[a,mb]$ if and only if $fog$ is $(s,m)$-convex in second sense on $[a,mb]$.
\end{proposition}
\begin{proof}
Since
\begin{equation}\label{P2}
 (fog) ( ta + m(1 - t)b) = f \big( \frac{mab}{mbt + (1 - t)a}\big)
 \end{equation}
 for all $t \in [0,1]$, $m \in (0, 1]$. The proof is obvious from equality (\ref{P2}).
%
%
\end{proof}
The following result of the Hermite-Hadamard type holds.
\begin{theorem}\label{II1}
Let $f:I \subset (0,\infty) \rightarrow \mathbb{R}$ be a harmonically $(s,m)$-convex function in second sense with $s \in [0,1]$ and $m \in (0,1]$. If $ 0 < a < b < \infty$ and $f \in L[a,b]$, then one has following inequality

$$\frac{ab}{b - a} \int _{a}^{b} \frac{f(x)}{x^{2}} dx \leq \min \vast[ \frac{f(a) + m f(\frac{b}{m})}{s + 1}, \frac{f(b) + m f(\frac{a}{m})}{s + 1} \vast]    $$

\end{theorem}
\begin{proof}
Since, $f:I \subset (0,\infty) \rightarrow \mathbb{R}$ is a harmonically $(s,m)$-convex function in second sense. We have, for $x,y \in I \subset (0,\infty)$
$$f \big( \frac{xy}{ty + (1 - t)x}\big) =    f \big( \frac{mx\frac{y}{m}}{mt\frac{y}{m} + (1 - t)x}\big) \leq t^{s} f(x) + m(1 - t)^{s}f(\frac{y}{m})   $$
which gives
$$f \big( \frac{ab}{tb + (1 - t)a}\big) \leq t^{s} f(a) + m(1 - t)^{s}f(\frac{b}{m})   $$
and
$$f \big( \frac{ab}{ta + (1 - t)b}\big) \leq t^{s} f(b) + m(1 - t)^{s}f(\frac{a}{m})   $$
for all $t \in [0,1]$. Integrating on $[0,1]$ w.r.t 't', we obtain
$$\int_{0}^{1} f\big( \frac{ab}{tb + (1 - t)a} \big)dt \leq \frac{f(a) + m f(\frac{b}{m})}{s + 1}       $$
and
$$\int_{0}^{1} f\big( \frac{ab}{ta + (1 - t)b} \big)dt \leq \frac{f(b) + m f(\frac{a}{m})}{s + 1}       $$

However,

$$\int_{0}^{1} f\big( \frac{ab}{tb + (1 - t)a} \big)dt = \int_{0}^{1} f\big( \frac{ab}{tb + (1 - t)a} \big)dt  = \frac{ab}{b - a} \int _{a}^{b} \frac{f(x)}{x^{2}} dx  $$

Hence, required inequality is established.
\end{proof}
\begin{corollary}
If we take $m = 1$ in theorem \ref{II1}, then we get
$$\frac{ab}{b - a} \int _{a}^{b} \frac{f(x)}{x^{2}} dx \leq   \frac{f(a) +  f(b)}{s + 1}    $$

\end{corollary}
\begin{corollary}
If we take $s = 1$ in theorem \ref{II1}, then we get
$$\frac{ab}{b - a} \int _{a}^{b} \frac{f(x)}{x^{2}} dx \leq   \min \vast[ \frac{f(a) + m f(\frac{b}{m})}{2}, \frac{f(b) + m f(\frac{a}{m})}{2} \vast]$$
\end{corollary}

\section{\textbf{Main Results}}

For finding some new inequalities of Hermite-Hadamard type for the functions whose derivatives are harmonically $(s,m)$-convex in second sense, we need the following lemma
\begin{lemma}
Let $f : I \subset \mathbb{R} / \{0\} \rightarrow \mathbb{R}$ be a differentiable function on $I^{\circ}$ with $a < b$. If $f \in L[a,b]$, then
$$ \frac{f(a) +  f(b)}{2}  - \frac{ab}{b - a} \int _{a}^{b} \frac{f(x)}{x^{2}} dx = \frac{ab(b - a)}{2} \int_{0}^{1} \frac{1 - 2t}{(tb + (1 - t)a)} f' \big( \frac{ab}{(tb + (1 - t)a)}\big)dt    $$

\end{lemma}

\begin{theorem}\label{II2}
Let $f :{I \subset (0,\infty)}   \rightarrow \mathbb{R}$ be a differentiable function on $I^{\circ}$, $a,\frac{b}{m} \in I^{\circ}$ with $a < b$, $m \in (0,1]$ and $f' \in L[a,b]$. If $|f'|^{q}$ is harmonically $(s,m)$-convex in second sense on $[a,\frac{b}{m}]$ for $q \geq 1$ with $s \in [0,1]$, then

$$\vast| \frac{f(a) + f(b)}{2} - \frac{ab}{b - a} \int_{a}^{b} \frac{f(x)}{x^{2}}dx \vast| \leq \frac{ab(b - a)}{2^{2 - \frac{1}{q}}} \big [ \rho_{1} (s,q;a,b)|f'(a)|^{q} + m \rho_{2}(s,q;a,b)|f'(\frac{b}{m})|^{q} \big]^{\frac{1}{q}}   $$
where
 \begin{eqnarray*}
 \rho_{1} (s,q;a,b) &=& \frac{\beta(1, s + 2)}{b^{2q}}._{2}F_{1} \big( 2q,1;s + 3 ;1 - \frac{a}{b} \big) - \frac{\beta(2,s + 1)}{b^{2q}}._{2}F_{1} \big( 2q,2;s + 3; 1 - \frac{a}{b} \big)\\
 &+&\frac{2^{2q - s} \beta(2,s + 1)}{(a + b)^{2q}} ._{2}F_{1} \big( 2q,2;s + 3; 1 - \frac{2a}{a + b} \big)
 \end{eqnarray*}
 \begin{eqnarray*}
 \rho_{2}(s,q;a,b) &=& \frac{\beta(s + 1,2)}{2^{s}b^{2q}}._{2}F_{1} \big(2q,s + 1;s + 3,\frac{1}{2}(1 - \frac{a}{b}) \big) - \frac{\beta(s + 1,2)}{b^{2q}}._{2}F_{1} \big(2q,s + 1;s + 3,1 - \frac{a}{b} \big)\\
 &+& \frac{\beta(s + 2,1)}{b^{2q}}._{2}F_{1} \big(2q,s + 2;s + 3,1 - \frac{a}{b} \big)
 \end{eqnarray*}
 $\beta$ is Euler Beta function defined by

 $$ \beta(x,y) = \frac{\Gamma(x) \Gamma(y)}{\Gamma(x + y)} = \int_{0}^{1} t^{x - 1} (1 - t)^{y - 1} dt,\; x,y>0 $$
 and$ _{2}F_{1}$ is hypergeometric function defined by
 $$_{2}F_{1}(a,b;c,z) = \frac{1}{\beta(b,c - b)}\int_{0}^{1} t^{b - 1} (1 - t)^{c - b -1} (1 - zt)^{-a} dt,\; c>b>0,\;|z|<1  $$
\end{theorem}
\begin{proof}
From above Lemma and using power mean inequality, we have
\begin{eqnarray*}
\vast| \frac{f(a) + f(b)}{2} - \frac{ab}{b - a} \int_{a}^{b} \frac{f(x)}{x^{2}}dx \vast| &\leq& \frac{ab(b - a)}{2} \int_{0}^{1} \big| \frac{1 - 2t}{(tb + (1 - t)a)^{2}}\big| \big| f' \big( \frac{ab}{tb + (1 - t)a} \big) \big|dt\\
&\leq& \frac{ab(b - a)}{2} \vast( \int_{0}^{1} |1 - 2t|dt \vast)^{1 - \frac{1}{q}}\\
&\times& \vast(   \int_{0}^{1} \frac{|1 - 2t|}{(tb + (1 - t)a)^{2q}} \big| f' \big( \frac{ab}{tb + (1 - t)a} \big) \big|^{q}dt \vast)^{\frac{1}{q}}
\end{eqnarray*}
Since, $|f'|^{q}$ is harmonically $(s,m)$-convex function in second sense, we have
$$ \vast| \frac{f(a) + f(b)}{2} - \frac{ab}{b - a} \int_{a}^{b} \frac{f(x)}{x^{2}}dx \vast|  $$
$$\leq \frac{ab(b - a)}{2} \big( \frac{1}{2} \big)^{1 - \frac{1}{q}}\vast( \int_{0}^{1} \frac{|1 - 2t|[t^{s} |f'(a)|^{q} + m(1 - t)^{s}|f'(\frac{b}{m})|^{q}]}{(tb + (1 - t)a)^{2q}}dt \vast)^{\frac{1}{q}} $$

$$ = \frac{ab(b - a)}{2} \big( \frac{1}{2} \big)^{1 - \frac{1}{q}}\vast[ |f'(a)|^{q} \int_{0}^{1}\frac{|1 - 2t|t^{s}}{(tb + (1 - t)a)^{2q}}dt + m |f'(\frac{b}{m})|^{q}\int_{0}^{1}\frac{|1 - 2t|(1 - t)^{s}}{(tb + (1 - t)a)^{2q}}dt  \vast]^{\frac{1}{q}} $$
$$ = \frac{ab(b - a)}{2} \big( \frac{1}{2} \big)^{1 - \frac{1}{q}}\vast[\rho_{1}(s,q;a,b) |f'(a)|^{q}  + m \rho_{2}(s,q;a,b) |f'(\frac{b}{m})|^{q} \vast]^{\frac{1}{q}}   $$
It is easy to check that
\begin{eqnarray*}
\int_{0}^{1}\frac{|1 - 2t|t^{s}}{(tb + (1 - t)a)^{2q}}dt &=& \int_{0}^{\frac{1}{2}}\frac{|1 - 2t|t^{s}}{(tb + (1 - t)a)^{2q}}dt + \int_{\frac{1}{2}}^{1}\frac{|1 - 2t|t^{s}}{(tb + (1 - t)a)^{2q}}dt\\
& = &\frac{2^{2q - s} \beta(2, s+ 1)}{(a + b)^{2q}}._{2}F_{1} (2q,2;s + 3, 1 - \frac{2a}{a + b})
-\frac{\beta(2,s  + 1)}{b^{2q}}._{2}F_{1}(2q,2;s + 3,1 - \frac{a}{b})\\
&+& \frac{\beta(1,s + 2)}{b^{2q}}._{2}F_{1}(2q,1;s + 3,1 - \frac{a}{b})\\ &:=& \rho_{1} (s,q;a,b)
\end{eqnarray*}
and
\begin{eqnarray*}
\int_{0}^{1}\frac{|1 - 2t|(1 - t)^{s}}{(tb + (1 - t)a)^{2q}}dt
&=& \frac{\beta(s + 1,2)}{2^{s}b^{2q}}._{2}F_{1}(2q,s + 1;s + 3,\frac{1}{2}(1 - \frac{a}{b})) -\frac{\beta(s + 2,1)}{b^{2q}}._{2}F_{1}(2q,s + 2;s + 3,(1 - \frac{a}{b}))\\
&+&\frac{\beta(s + 2,1)}{b^{2q}}._{2}F_{1}(2q,s + 2;s + 3,(1 - \frac{a}{b})) := \rho_{2}(s,q;a,b)
\end{eqnarray*}
This completes the proof.
\end{proof}
If we take $s = m = 1$ in Theorem \ref{II2}, we get the following
\begin{corollary}
Let $f:I \subset (0,\infty) \rightarrow \mathbb{R}$ be differentiable function on $I\circ$, $a,b \in I\circ$ with $a < b$ and $f' \in L[a,b]$. If $|f'|^{q}$ is $(1,1)$-harmonically convex in second sense or harmonically convex function on $[a,b]$ for $q \geq 1$, then
$$ \vast| \frac{f(a) + f(b)}{2} - \frac{ab}{b - a} \int_{a}^{b} \frac{f(x)}{x^{2}}dx \vast| \leq \frac{ab(b - a)}{2^{2 - \frac{1}{q}}} \big [ \rho_{1} (1,q;a,b)|f'(a)|^{q} +  \rho_{2}(1,q;a,b)|f'(b)|^{q} \big]^{\frac{1}{q}}    $$
\end{corollary}
\begin{corollary}
If we take $m = 1$ in Theorem \ref{II2}, then we get
$$\vast| \frac{f(a) + f(b)}{2} - \frac{ab}{b - a} \int_{a}^{b} \frac{f(x)}{x^{2}}dx \vast| \leq \frac{ab(b - a)}{2^{2 - \frac{1}{q}}} \big [ \rho_{1} (s,q;a,b)|f'(a)|^{q} +  \rho_{2}(s,q;a,b)|f'(b)|^{q} \big]^{\frac{1}{q}}$$
\end{corollary}
\begin{corollary}
If we take $s = 1$ in Theorem \ref{II2}, we get
$$\vast| \frac{f(a) + f(b)}{2} - \frac{ab}{b - a} \int_{a}^{b} \frac{f(x)}{x^{2}}dx \vast| \leq \frac{ab(b - a)}{2^{2 - \frac{1}{q}}} \big [ \rho_{1} (1,q;a,b)|f'(a)|^{q} + m \rho_{2}(1,q;a,b)|f'(\frac{b}{m})|^{q} \big]^{\frac{1}{q}}   $$

\end{corollary}
\begin{theorem}\label{II3}
Let $f :{I \subset (0,\infty)}   \rightarrow \mathbb{R}$ be a differentiable function on $I$, $ma,b \in I^{\circ}$ with $a < b$, $m \in (0,1]$ and $f' \in L[a,b]$. If $|f'|^{q}$ is harmonically $(s,m)$-convex in second sense on $[a,\frac{b}{m}]$ for $q \geq 1$ with $s \in [0,1]$, then

$$\vast| \frac{f(a) + f(b)}{2} - \frac{ab}{b - a} \int_{a}^{b} \frac{f(x)}{x^{2}}dx \vast|$$
$$\leq \frac{ab(b - a)}{2}\rho_{1}^{1 - \frac{1}{q}}(0,q;a,b) \big [ \rho_{1} (s,q;a,b)|f'(a)|^{q} + m \rho_{2}(s,q;a,b)|f'(\frac{b}{m})|^{q} \big]^{\frac{1}{q}}   $$
\end{theorem}
\begin{proof}
From Lemma, Power mean inequality and harmonically $(s,m)$-convexity in second sense of$ |f'|^{q}$ on $[a,\frac{b}{m}]$,we have
$$\vast| \frac{f(a) + f(b)}{2} - \frac{ab}{b - a} \int_{a}^{b} \frac{f(x)}{x^{2}}dx \vast| $$
$$\leq \frac{ab(b - a)}{2} \int_{0}^{1} \big| \frac{1 - 2t}{(tb + (1 - t)a)^{2}}\big| \big| f' \big( \frac{ab}{tb + (1 - t)a} \big) \big|dt  $$
$$\leq \frac{ab(b - a)}{2} \vast( \int_{0}^{1} |\frac{1 - 2t}{(tb + (1 - t)a)^{2}}|dt \vast)^{1 - \frac{1}{q}}
\vast(   \int_{0}^{1} \frac{|1 - 2t|}{(tb + (1 - t)a)^{2}} \big| f' \big( \frac{ab}{tb + (1 - t)a} \big) \big|^{q}dt \vast)^{\frac{1}{q}}   $$
$$\leq \frac{ab(b - a)}{2} \vast( \int_{0}^{1} |\frac{1 - 2t}{(tb + (1 - t)a)^{2}}|dt \vast)^{1 - \frac{1}{q}}\vast( \int_{0}^{1} \frac{|1 - 2t|[t^{s} |f'(a)|^{q} + m(1 - t)^{s}|f'(\frac{b}{m})|^{q}]}{(tb + (1 - t)a)^{2}}dt \vast)^{\frac{1}{q}}    $$
$$\leq \frac{ab(b - a)}{2}\rho_{1}^{1 - \frac{1}{q}}(0,q;a,b) \big [ \rho_{1} (s,q;a,b)|f'(a)|^{q} + m \rho_{2}(s,q;a,b)|f'(\frac{b}{m})|^{q} \big]^{\frac{1}{q}}   $$
\end{proof}
\begin{corollary}
If we take $m = 1$ in Theorem \ref{II3}, then we get
$$\vast| \frac{f(a) + f(b)}{2} - \frac{ab}{b - a} \int_{a}^{b} \frac{f(x)}{x^{2}}dx \vast|$$
$$\leq \frac{ab(b - a)}{2}\rho_{1}^{1 - \frac{1}{q}}(0,q;a,b) \big [ \rho_{1} (s,q;a,b)|f'(a)|^{q} +  \rho_{2}(s,q;a,b)|f'(b)|^{q} \big]^{\frac{1}{q}}   $$
This is Theorem \ref{FS1} proved by Feixiang Chen and Shanhe Wu in \cite{FU}.
\end{corollary}
\begin{corollary}
If we take $s = 1$ in Theorem \ref{II3}, we get
$$\vast| \frac{f(a) + f(b)}{2} - \frac{ab}{b - a} \int_{a}^{b} \frac{f(x)}{x^{2}}dx \vast|$$
$$\leq \frac{ab(b - a)}{2}\rho_{1}^{1 - \frac{1}{q}}(0,q;a,b) \big [ \rho_{1} (1,q;a,b)|f'(a)|^{q} + m \rho_{2}(1,q;a,b)|f'(\frac{b}{m})|^{q} \big]^{\frac{1}{q}}   $$
\end{corollary}
\begin{corollary}
If we take $s = m = 1$ in Theorem \ref{II3}, we get
$$\vast| \frac{f(a) + f(b)}{2} - \frac{ab}{b - a} \int_{a}^{b} \frac{f(x)}{x^{2}}dx \vast|$$
$$\leq \frac{ab(b - a)}{2}\rho_{1}^{1 - \frac{1}{q}}(0,q;a,b) \big [ \rho_{1} (1,q;a,b)|f'(a)|^{q} +  \rho_{2}(1,q;a,b)|f'(b)|^{q} \big]^{\frac{1}{q}}   $$
which is Theorem \ref{I1} proved by Imdat Iscan in \cite{Ak}.
\end{corollary}
\begin{theorem}\label{II4}
Let $f :{I \subset (0,\infty)}   \rightarrow \mathbb{R}$ be a differentiable function on $I^{\circ}$, $ma,b \in I^{\circ}$ with $a < b$, $m \in (0,1]$ and $f' \in L[a,b]$. If $|f'|^{q}$ is harmonically $(s,m)$-convex in second sense on $[a,\frac{b}{m}]$ for $q > 1$, $ \frac{1}{p} + \frac{1}{q} = 1$ with $s \in [0,1]$, then
$$\vast| \frac{f(a) + f(b)}{2} - \frac{ab}{b - a} \int_{a}^{b} \frac{f(x)}{x^{2}}dx \vast|$$
$$\leq \frac{ab(b - a)}{2} \big( \frac{1}{p + 1} \big)^{\frac{1}{p}} \big[ \nu_{1}(s,q;a,b)|f'(a)|^{q} + m \nu_{2}(s,q;a,b)|f'(\frac{b}{m})|^{q} \big]^{\frac{1}{q}}$$

where

$$ \nu_{1}(s,q;a,b) = \frac{\beta(1,s + 1)}{b^{2q}}._{2}F_{1}(2q,1;s + 2,1 - \frac{a}{b})$$
and
$$   \nu_{2}(s,q;a,b) = \frac{\beta(s + 1,1)}{b^{2q}}._{2}F_{1}(2q,s + 1;s + 2,1 - \frac{a}{b})  $$
\end{theorem}
\begin{proof}
From Lemma, H\"{o}lder's inequality and harmonically $(s,m)$-convexity of $|f'|^{q}$ on $[a,\frac{b}{m}]$, we have
$$\vast| \frac{f(a) + f(b)}{2} - \frac{ab}{b - a} \int_{a}^{b} \frac{f(x)}{x^{2}}dx \vast| $$
$$\leq \frac{ab(b - a)}{2}\vast( \int_{0}^{1}|1 - 2t|^{p}dt\vast)^{\frac{1}{p}} \vast( \int_{0}^{1} \frac{1}{(tb + (1 - t)a)^{2q}}\big| f' \big( \frac{ab}{tb + (1 - t)a} \big) \big|dt  $$
$$\leq \frac{ab(b - a)}{2} \vast(\frac{1}{p + 1} \vast)^{\frac{1}{p}} \vast[ |f'(a)|^{q} \int_{0}^{1} \frac{t^{s}}{(tb + (1 - t)a)^{2q}}dt + m|f'(\frac{b}{m})|^{q} \int_{0}^{1} \frac{ (1 - t)^{s}}{(tb + (1 - t)a)^{2q}}dt \vast]^{\frac{1}{q}} $$
$$= \frac{ab(b - a)}{2} \vast(\frac{1}{p + 1} \vast)^{\frac{1}{p}} \vast[ |f'(a)|^{q} \nu_{1}(s,q;a,b) + m|f'(\frac{b}{m})|^{q} \nu_{2}(s,q;a,b) \vast]^{\frac{1}{q}} $$
where an easy calculation gives
\begin{eqnarray*}
\int_{0}^{1} |1 - 2t|^{p} dt
&=&\frac{1}{p + 1}
\end{eqnarray*}
\begin{eqnarray*}
\int_{0}^{1} \frac{t^{s}}{(tb + (1 - t)a)^{2q}}dt
&=& \frac{\beta(1,s + 1)}{b^{2q}}._{2}F_{1}(2q,1;s + 2,1 - \frac{a}{b}) := \nu_{1}(s,q;a,b)
\end{eqnarray*}
and
\begin{eqnarray*}
\int_{0}^{1} \frac{(1 - t)^{s}}{(tb + (1 - t)a)^{2q}}dt
&=& \frac{\beta(s + 1,1)}{b^{2q}}._{2}F_{1}(2q,s = 1;s + 2,1 - \frac{a}{b}) := \nu_{2}(s,q;a,b)
\end{eqnarray*}
This completes the proof.
\end{proof}
\begin{corollary}
If we take $ m = 1$ in Theorem \ref{II4}, then we get
$$\vast| \frac{f(a) + f(b)}{2} - \frac{ab}{b - a} \int_{a}^{b} \frac{f(x)}{x^{2}}dx \vast|$$
$$\leq \frac{ab(b - a)}{2} \big( \frac{1}{p + 1} \big)^{\frac{1}{p}} \big[ \nu_{1}(s,q;a,b)|f'(a)|^{q} +  \nu_{2}(s,q;a,b)|f'(b)|^{q} \big]^{\frac{1}{q}}$$
this is Theorem \ref{FS2} proved by Feixiang Chen and Shanhe Wu in \cite{FU}.
\end{corollary}
\begin{corollary}
If we take $s = 1$ in above Theorem, then we get
$$\vast| \frac{f(a) + f(b)}{2} - \frac{ab}{b - a} \int_{a}^{b} \frac{f(x)}{x^{2}}dx \vast|$$
$$\leq \frac{ab(b - a)}{2} \big( \frac{1}{p + 1} \big)^{\frac{1}{p}} \big[ \nu_{1}(1,q;a,b)|f'(a)|^{q} + m \nu_{2}(1,q;a,b)|f'(\frac{b}{m})|^{q} \big]^{\frac{1}{q}}$$

\end{corollary}
\begin{corollary}
If we take $s = m = 1$ in above Theorem, then we get
$$\vast| \frac{f(a) + f(b)}{2} - \frac{ab}{b - a} \int_{a}^{b} \frac{f(x)}{x^{2}}dx \vast|$$
$$\leq \frac{ab(b - a)}{2} \big( \frac{1}{p + 1} \big)^{\frac{1}{p}} \big[ \nu_{1}(1,q;a,b)|f'(a)|^{q} +  \nu_{2}(1,q;a,b)|f'(b)|^{q} \big]^{\frac{1}{q}}$$
This is Theorem \ref{I2} proved by Imdat Iscan in \cite{Ak}.
\end{corollary}

\end{document}